\documentclass[11pt]{amsart}
\usepackage[margin=1.2in]{geometry}
\usepackage{amsmath}
\usepackage{amsthm}
\usepackage{color}
\usepackage{amsfonts}
\usepackage{amssymb}
\usepackage{amscd}
\numberwithin{equation}{section}
\newtheorem{theorem}{Theorem}[section]
\newtheorem{cor}[theorem]{Corollary}
\newtheorem{conj}[theorem]{Conjecture}
\newtheorem{remark}[theorem]{Remark}
\newtheorem{proposition}[theorem]{Proposition}
\newtheorem{prop}[theorem]{Proposition}
\newtheorem{lemma}[theorem]{Lemma}
\theoremstyle{definition}
\newtheorem{defi}{Definition}[section]
\newtheorem{rem}[defi]{Remark}

\newtheorem*{thma}{Theorem}

\newcommand\half{\frac{1}{2}}

\newcommand\g{\mathfrak g}
\newcommand\ga{\widehat{\mathfrak g}}
\newcommand\h{\mathfrak h}

\newcommand\D{\Delta}
\renewcommand\l{\lambda}
\newcommand\Dp{\Delta^+}

\renewcommand\d{\delta}
\renewcommand\r{\mathfrak r}

\renewcommand\a{\alpha}
\renewcommand\aa{\mathfrak a}

\renewcommand\k{\mathfrak k}

\newcommand\ganz{\mathbb Z}

\newcommand\s{\sigma}
\renewcommand\L{\Lambda}

\renewcommand\aa{\mathfrak a}

\newcommand\C{\mathbb C}

\newcommand\p{\mathfrak p}

\newcommand{\vac}{{\bf 1}}
\newcommand{\bea}{\begin{eqnarray}}
\newcommand{\eea}{\end{eqnarray}}
\begin{document}
\title[Kostant's pair of Lie type and conformal embeddings]{Kostant's pair of Lie type and conformal embeddings}
\author[Adamovi\'c, Kac, M\"oseneder, Papi, Per\v{s}e]{Dra{\v z}en~Adamovi\'c}
\author[]{Victor~G. Kac}
\author[]{Pierluigi M\"oseneder Frajria}
\author[]{Paolo  Papi}
\author[]{Ozren  Per\v{s}e}

\begin{abstract} We deal with some  aspects of the theory of conformal embeddings of affine vertex algebras, providing a 
new proof of the Symmetric Space  Theorem and  a  criterion for conformal embeddings of equal rank subalgebras.
We finally study some  examples of embeddings at the critical level. We prove a criterion for embeddings at the critical level which enables us to prove equality of certain central elements.
\end{abstract}
\maketitle

\section{Introduction}
In this paper we deal with three aspects related to the notion of conformal embedding of affine vertex algebras. We provide the following three contributions:
\begin{enumerate}
\item  a proof of the Symmetric Space  Theorem via Kostant's theory of {\it pairs of Lie type};
\item a central charge criterion for conformal embeddings of equal rank subalgebras;
\item a detailed  study of some relevant examples of embeddings at the critical level.
\end{enumerate}

To illustrate  the previous topics, let us recall that if  
$\g$ is a semisimple finite-dimensional complex Lie algebra and $\k\subset \g$  a reductive subalgebra, the embedding $\k\hookrightarrow \g$ is  called conformal    if the central charge of the 
Sugawara construction for
the affinization $\ga$,  acting on a  integrable $\ga$-module of level $k$, equals that for $\widehat{\k}$. Then necessarily $k=1$ \cite{AGO}.
Maximal conformal embeddings were classified
in \cite{SW},  \cite{AGO}, and related decompositions can be found in \cite{KWW}, \cite{KS}, \cite{CKMP}.
 In the  vertex algebra framework the definition can be rephrased as follows: the simple affine vertex  algebras $V_1(\g)$ and its  vertex subalgebra  $\widetilde V(k,\k)$  generated by $\{x_{(-1)}\vac\mid x\in\k\}$ have the same
Sugawara conformal vector.\par
In \cite{AKMPP1} we generalized the previous situation to study when the simple affine vertex  algebra $V_k(\g)$ and its subalgebra of $\widetilde V(k,\k)$ have the same
Sugawara conformal vector  for some  non-critical level $k$, not necessarily $1$, assuming  $\k$ to be  a maximal  equal rank reductive subalgebra. The non equal rank case has been studied in \cite{AKMPP3}.

The conformal embeddings in $V_k(so(V))$ with $k\in \ganz_+$ and $V$ a finite dimensional representation are the subject of the Symmetric Space Theorem (see Theorem \ref{SST}), which is at the basis of the early history of the theory. 
In our  subsequent  work on conformal embeddings  (see in particular \cite[Section 3]{AKMPP3}) we used a special case of it, but  we later realized that it is indeed possible to use  Kostant's theory quoted above (together with other tools from the theory of affine and vertex algebras) to provide a new proof of the theorem in full generality. This is done in Section \ref{KT}.
\par
In Section \ref{ERE} we improve our results from \cite{AKMPP1}, where we mainly studied conformal embeddings of maximal equal rank subalgebras, by providing full details on a criterion  for conformal embedding in which we drop the maximality assumption.

The final section  is devoted to certain instances of the following general problem: 
 provide  the explicit decomposition of $V_k(\g)$ regarded as a $\widetilde V(k,\k)$-module.  This problem was studied in \cite{AKMPP1} and \cite{AKMPP3}. When $k$ is non-integral, it was noted  that the following situations can  occur:
\begin{itemize}
\item $V_k(\g)$ is a semisimple $\widetilde V(k,\k)$--module with finite or infinite decompositions.
\item  $\widetilde V(k,\k)$ is not simple and  $V_k(\g)$ contains subsingular vectors.
\item $\widetilde V(k,\k)$  is a vertex algebra at the critical level.
\end{itemize}
There are many examples for which  explicit decompositions are still unknown. Such cases are also connected with some recent physics paper and conjectures. 
In Section \ref{CL} we  present some results in the case when $\widetilde V(k,\k)$  is a vertex algebra at the critical level. We show that then $\widetilde V(k,\k)$ is never simple by constructing certain non-trivial central elements in $\widetilde V(k,\k)$.  In most cases,   the embedded vertex subalgebra is  a certain quotient of
 $( V^{-h^{\vee} } ({\g}_1 ) \otimes  V^{-h^{\vee}} ({\g}_2)    ) \slash   \langle s  \rangle $ where $s$ is a linear combination  of quadratic Sugawara central elements.

 \vskip 5mm
{\bf Acknowledgment.}
D.A. and O. P. are  partially supported by the Croatian Science Foundation under the project 2634 and by the
QuantiXLie Centre of Excellence, a project cofinanced
by the Croatian Government and European Union
through the European Regional Development Fund - the
Competitiveness and Cohesion Operational Programme
(KK.01.1.1.01).

\section{Conformal Embeddings}
 Let $\g$ be  a simple Lie algebra. Let $\h$ be a Cartan subalgebra, $\D$ the 
$(\g,\h)$-root system, $\Dp$  a set of positive roots and $\rho$ the corresponding Weyl vector. Let $(\cdot,\cdot)$ denote the   normalized bilinear invariant  form (i.e., $(\a,\a)=2$ for any long root).

We denote by $V^k(\g)$ the universal affine vertex algebra of level $k$. If $k+h^\vee\ne 0$ then $V^k(\g)$ admits a unique simple quotient that we denote by $V_k(\g)$. More generally, if  $\aa$ is  a reductive Lie algebra that decomposes as  $\aa=\aa_0\oplus\dots\oplus \aa_s$ with $\aa_0$ abelian and $\aa_i$ simple ideals for $i>0$ and $\mathbf{k}=(k_1,\ldots,k_s)$ is a multi-index of levels, we let
$$
V^{\mathbf{k}}(\aa)= V^{k_0}(\aa_0)\otimes\dots\otimes V^{k_s}(\aa_s),\quad
V_{\mathbf{k}}(\aa)= V_{k_0}(\aa_0)\otimes\dots\otimes V_{k_s}(\aa_s).
$$
We let $\vac$ denote the vacuum vector of both $V^{\mathbf{k}}(\aa)$ and $V_{\mathbf{k}}(\aa)$.

If $j>0$, let $\{x_i^j\},\{y_i^j\}$ be dual bases of $\aa_j$ with respect to the normalized invariant form of $\aa_j$ and $h^\vee_j$  its dual Coxeter number. For $\aa_0$, let $\{x_i^j\},\{y_i^j\}$ be dual bases with respect to any chosen nondegenerate form and set $h_0^\vee=0$.\par
Assuming that $k_j+h_j^\vee\ne 0$ for all $j$, we consider  $V^{\mathbf{k}}(\aa)$ and all its quotients, including $V_{\mathbf{k}}(\aa)$, as  conformal vertex algebras with conformal vector $\omega_\aa$ given by the Sugawara construction:
$$\omega_\aa=\sum_{j=0}^s\frac{1}{2(k_j+h_j^\vee)}\sum_{i=1}^{\dim \aa_j} :x^j_i y_i^j :.
$$

Recall that the central charge of $\omega_\g$
 for a simple or abelian Lie algebra $\g$ is, assuming $k+h^\vee\ne0$,
\begin{equation}\label{cc}
c(\g,k)=\frac{k\dim\g}{k+h^\vee}.
\end{equation}
If $\aa$ is a reductive Lie algebra, which decomposes as  $\aa=\aa_0\oplus\dots\oplus \aa_s$ then the central charge of $\omega_\aa$ is, for a multindex $\mathbf{k}=(k_0,\ldots,k_s)$,
\begin{equation}\label{ccc}
c(\aa,\mathbf{k})=\sum\limits_{j=0}^s c(\aa_j,k_j).
\end{equation}

\subsection{AP-criterion}\label{APC}
Assume that $\k$ is a Lie   subalgebra in a simple Lie algebra $\g$ which is reductive in $\g$. Then $\k$ decomposes as 
\begin{equation*}\label{decompg0}
\k=\k_0\oplus\cdots\oplus \k_t.
\end{equation*}
where 
$\k_1,\ldots \k_t$ are the simple ideals of $\k$ and $\k_0$ is the center of $\k$. 
Let $\p$ be  the orthocomplement of $\k$ in $\g$ w.r.t to $(\cdot, \cdot)$ and let 
$$\p=\bigoplus_{i=1}^N\left(\bigotimes_{j=1}^{t} V(\mu_i^j)\right)$$ be its  decomposition  as a $\k$-module.  Let  $( \cdot, \cdot)_j$ denote the normalized invariant
bilinear form on $\k_j$.  

We denote by $\widetilde V(k,\k)$ the vertex subalgebra of $V_{k}(\g)$ generated by $\{x_{(-1)}\vac\mid x\in\k\}$. Note that $\widetilde V(k,\k)$ is an affine vertex algebra, more precisely it is a quotient of $\otimes V^{k_j}(\k_j)$, with the levels $k_j$ determined by $k$ and the ratio between $(\cdot,\cdot)$ and $(\cdot,\cdot)_j$.

\begin{thma}(AP-criterion) \label{AP}\cite{A} {\it
 $ \widetilde{V}(k,\k)$ is
conformally embedded in $V_k({\g})$ if and only if
\bea\label{numcheck} && \sum_{j=0}^t\frac{ ( \mu_{i}^j , \mu _{i}^j + 2 \rho_0^j ) _j}{ 2 (k_j + h_j
^{\vee} )} = 1  \eea
for any $i=1,\ldots,s$.}
\end{thma}

We reformulate the criterion highlighting  the dependence from the choice of the form on $\k$.  Fix  an invariant nondegenerate symmetric form $B_\k$  on $\k$. Then $(\cdot,\cdot)_{|\k_j\times \k_j} =n_j(B_\k)_{|\k_j\times \k_j}$.  Let $\{X_i\}$ be an orthonormal basis  of $\k_j$ with respect to $B_\k$ and let $C_{\k_j}= \sum_i X_i^2$ be the corresponding Casimir operator. Let $2 g^{B_\k}_j$ be the eigenvalue for the action of $C_{\k_j}$ on $\k_j$ and $\gamma_{ji}^{B_\k}$ the eigenvalue of the action of $C_{\k_j}$ on $V(\mu^j_i)$. Then  $ \widetilde{V}_{k}(\k)$ is
conformally embedded in $V_k({\g})$ if and only if
\bea\label{numcheck2} && \sum_{j=0}^t\frac{ \gamma_{ji}^{B_\k}}{ 2 (n_jk+ g_j^{B_\k} )} = 1  \eea
for any $i=1,\ldots,s$.

\begin{cor}\label{Actsscalar}Assume $\k$ is simple, so that $\k=\k_1$. 
Then there is $k\in\C$ such that $\widetilde V(k,\k)$ is conformally embedded in $V_k(\g)$ if and only if $C_\k$ acts scalarly on $\p$.
\end{cor}
\begin{proof}If $\widetilde V(k,\k)$ is conformally embedded in $V_k(\g)$ then, 
by \eqref{numcheck2}, $\gamma_i^1=2(k+g_1)$ is independent of  $i$.
If $C_\k$ acts scalarly on $\p$, then, solving \eqref{numcheck2} for $k$, one finds a level where, by AP-criterion, conformal embedding occurs.
\end{proof}

\section{Symmetric Space Theorem}\label{KT}
The Symmetric Space Theorem has originally been proved in \cite{GNO}; an algebraic approach can be found in \cite{Dab}. We recall this theorem
in a vertex algebra formulation.

\begin{theorem}[{\bf Symmetric Space Theorem}]\label{SST}
Let $K$ be a compact, connected, non-trivial Lie group with complexified Lie algebra $\mathfrak k$ and let $\nu:K\to End(V)$ be a finite dimensional faithful representation admitting a $K$-invariant symmetric nondegenerate form.

Then there is a conformal  embedding of $\widetilde V(k,\mathfrak k)$ in $V_k(so(V))$ with $k\in\ganz_+$ if and only if $k=1$ and 
there is a Lie algebra structure on $\mathfrak r=\mathfrak k\oplus V$ making $\mathfrak r$ semisimple, $(\mathfrak r,\mathfrak k)$ is a symmetric pair, and a nondegenerate invariant form on $\mathfrak r$ is given by the direct sum of an invariant form on $\mathfrak k$ with the chosen $K$-invariant form on $V$.
\end{theorem}

In this section we give a proof of the above theorem using Kostant's theory of pairs of Lie type and Kac's {\it very strange formula} as main tools.

We start by showing that if $\mathfrak r=\mathfrak k\oplus V$ gives rise to a symmetric pair $(\mathfrak r,\mathfrak k)$, then $\widetilde V(1,\mathfrak k)$ embeds conformally in $V_1(so(V))$.
 
 It suffices to show that 
\begin{equation}\label{cequal}
\sum_jc(\k_j,k_j)=c(\g,k).
\end{equation}

In fact, the action of $(\omega_\g-\omega_\k)_{(n)}$ on $V_1(so(V))$ defines a unitary representation of the Virasoro algebra. The Virasoro algebra has no nontrivial unitary representations with zero central charge, hence, if \eqref{cequal} holds, then $(\omega_\g-\omega_\k)_{(n)}$ acts trivially on $V_1(so(V))$ for all $n$. Since $V_1(so(V))$ is simple, this implies that $\omega_\g=\omega_\k$.

The normalized invariant form on $so(V)$ is given by $(X,Y)=\half tr(XY)$. Let $\k_{re}$ be the Lie algebra of $K$. Since $V$ admits a nondegenerate symmetric invariant form, there is a real subspace $U$ of $V$ which is $K$-stable and $V=U\oplus \sqrt{-1}U$. Choose an orthonormal basis $\{u_i\}$ of $U$ with respect to a $K$-invariant inner product on $U$. Then the matrix $\pi(X)$ of the action of $X\in\k_{re}$  on $V$ in this basis is real and antisymmetric. It follows that $tr(\pi(X)^2)\ne 0$ unless $X=0$, 
thus $(\cdot,\cdot)_{|\k\times \k}$ is nondegenerate.

Using $(\cdot,\cdot)_{|\k\times\k}$ as invariant form on $\k$ we have
$$
\sum_jc(\k_j,k_j)=\sum_j\frac{\dim\k_j}{1+g_j}
$$
while it is easy to check that $c(so(V),1)=\half\dim V$. Thus we need to check that
\begin{equation}\label{t1}
\sum_j\frac{\dim\k_j}{1+g_j}-\half\dim V=0.
\end{equation}

Let $\sigma$ denote the involution defining the pair. We can  assume that $\sigma$ is indecomposable, for, otherwise, one can restrict to each indecomposable ideal.



Let 
 $\kappa(\cdot,\cdot)$ be the Killing form on $\mathfrak r$. Since $\s$ is assumed indecomposable, the center $\k_0$ of $\k$ has dimension at most one. It follows that  $\kappa(\cdot,\cdot)_{|\k_j\times\k_j}=r_j(\cdot,\cdot)_{|\k_j\times\k_j}$ for some $r_j\in\C$, $j\ge 0$. Recall that  $(x,y)=\half tr_V(xy)$.
Note that,  
\begin{align*}
r_j\dim\k_j&=\sum_{i=1}^{\dim \k_j}\kappa(x^j_i,y^j_i)=2g_j\dim \k_j+\sum_{i=1}^{\dim \k_j}tr_V(x_i^jy_i^j)\\
&=2g_j\dim \k_j+\sum_{i=1}^{\dim \k_j}2(x_i^j,y_i^j)=2(g_j+1)\dim \k_j,
\end{align*}
hence $r_j=2(g_j+1)$. 

If we choose $B_\k=\kappa(\cdot,\cdot)_{|\k\times \k}$ as invariant form on $\k$ then $C^{B_\k}_{\k_j}=\frac{1}{r_j}C_{\k_j}$ so the eigenvalue of the action of $C^{B_\k}_{\k_j }$ on $\k_j$ is $\frac{g_j}{g_j+1}$.

We write now the {\it very strange}Ê formula  \cite[(12.3.6)]{Kac} in the form given in \cite[(55)]{Dop} for the  involution $\sigma$ using  the Killing form as invariant form on  $\mathfrak r$. Since $\sigma$ is an involution we have that $\rho_\sigma=\rho_\k$ and $z(\mathfrak r,\sigma)=\frac{\dim V}{16}$, hence the formula reads:
$$ \kappa(\rho_{\k},\rho_{\k})=
\frac{\dim\mathfrak r}{24}-\frac{\dim V}{16}, 
$$
On the other hand, using the strange formula for $\k$ (in the form of (54) in \cite{Dop}), we have
$$ \kappa(\rho_{\k},\rho_{\k})=
\sum_{j=1}^t\frac{g_j\dim\mathfrak k_j}{24(1+g_j)} 
$$
Equating these last two expressions one derives \eqref{t1}.

\vskip5pt
We now prove the converse implication: assume that there is a conformal  embedding of $\widetilde V(k,\mathfrak k)$ in $V_k(so(V))$ with $k\in\ganz_+$.
Let  $\nu:\k\to so(V)$ be the embedding.  Let also $\tau:\bigwedge^2 V \to so(V)$ be  the $\k$-equivariant isomorphism such that $\tau(u)(v)=-2i(v)(u)$, where $i$ is the contraction map, extended to $\bigwedge V$ as an odd derivation. More explicitly
$
\tau^{-1}(X)=\frac{1}{4}\sum_i X(v_i)\wedge v_i,
$
where $\{v_i\}$ is an orthonormal basis of $V$.

Choose an invariant nondegenerate symmetric form $B_{\k}$   on $\k$. Recall that $\p$ is the orthocomplement of $\k$ in $so(V)$ with respect to $(\cdot,\cdot)$. Let $p_\k$, $p_{\mathfrak p}$ be the orthogonal projections of $\bigwedge^2 V$ onto $\k$ and $\mathfrak p$ respectively.


We claim that we can choose $B_\k$ so that 
\begin{equation}\label{ennej}
2(n_jk+g^{B_\k}_j)=1
\end{equation} 
(see Section 2 for notation). Indeed, it is clear that  $n_jg_j=g_j^{B_\k}$, hence it is enough to choose 
$$
n_j=\frac{1}{2(k+g_j)}.
$$
With this choice for $B_\k$ we have that, by \eqref{numcheck2},
$$
\sum_j \gamma_{ji}^{B_\k}=1.
$$
In particular, $C^{B_\k}_\k$ acts on $\otimes_{j=1}^t V(\mu^j_i)$ as the identity for all $i$. Thus $C^{B_\k}_\k$ acts as the identity on $\p$.

Set $\nu_*=\tau^{-1}\circ\nu$ and let $Cl(V)$ denote the Clifford algebra of $(V,\langle\cdot,\cdot\rangle)$. Extend $\nu_*$ to a Lie algebra homomorphism $\nu_*: \k\to Cl(V)$, hence to a homomorphism of associative algebras  $\nu_*: U(\k)\to Cl(V)$.
  Consider  $\nu_*(C^{B_\k}_\k)=\sum_i\nu_*(X_i)^2$. Also recall that a pair $(\k,\nu)$ consisting of a 
Lie algebra $\k$ with a bilinear symmetric invariant form $B_\k$ and of a representation $\nu:\k\to so(V)$  is said to be 
 {\it of Lie type} if there is a Lie algebra structure on $\mathfrak r=\k\oplus V$, extending that of $\k$, such that 1) $[x,y]=\nu(x)y,\,x\in\k,\,y\in V$ and 2) the bilinear  form 
 $B_{\mathfrak r}=B_\k\oplus \langle \cdot, \cdot \rangle$ is $ad_{\mathfrak r}$-invariant. In \cite[Theorem 1.50]{Kost}  Kostant proved that $(\k,\nu)$ is a  pair of Lie type if and only if 
 there exists $v\in (\bigwedge^3\p)^\k$ such that 
 \begin{equation}\label{nu}\nu_*(C^{B_\k}_\k)+v^2\in\C.\end{equation} Moreover he proved in \cite[Theorem 1.59]{Kost} that $v$ can be taken to be $0$ if and only if $(\k\oplus V,\k)$ is a symmetric pair.
 
 \begin{lemma}\label{critKost}
Assume that $\widetilde V(k,\k)$ embeds conformally in $V_k(so(V))$ with $k\in \ganz_+$. 
Then
$$
\sum_i \tau^{-1}(X_i)\wedge \tau^{-1}(X_i)=0
$$
where $\{X_i\}$ is an orthonormal basis of $\k$ with respect to $B_\k$.
\end{lemma}
\begin{proof}
 Since $C^{B_\k}_\k$ is diagonalizable on $V$ and it is symmetric with respect to $\langle\cdot,\cdot\rangle$, there is an orthonormal basis $\{v_{i}\}$ of $V$ made of eigenvectors for $C^{B_\k}_\k$. Let $\lambda_i$ be the eigenvalue corresponding to $v_i$.
 Then 
\begin{align*}
\sum_i \tau^{-1}&(X_i)\wedge \tau^{-1}(X_i)=\frac{1}{16}\sum_{i,j,r} X_i(v_j)\wedge v_j\wedge X_i(v_r)\wedge v_r\\
&=-\frac{1}{16}\sum_{i,j,r} X_i(v_j)\wedge X_i(v_r)\wedge v_j\ \wedge v_r=-\frac{1}{32}\sum_{i,j,r} X^2_i(v_j\wedge v_r)\wedge v_j\ \wedge v_r\\&+\frac{1}{32}\sum_{i,j,r} X^2_i(v_j)\wedge v_r\wedge v_j\ \wedge v_r+\frac{1}{32}\sum_{i,j,r} v_j\wedge X^2_i( v_r)\wedge v_j\ \wedge v_r\\&=-\frac{1}{32}\sum_{i,j,r} X^2_i(v_j\wedge v_r)\wedge v_j\ \wedge v_r+\frac{1}{32}\sum_{j,r} \l_jv_j\wedge v_r\wedge v_j\ \wedge v_r\\&+\frac{1}{32}\sum_{j,r} v_j\wedge \l_r v_r\wedge v_j\ \wedge v_r=-\frac{1}{32}\sum_{i,j,r} X^2_i(v_j\wedge v_r)\wedge v_j\ \wedge v_r.
\end{align*}

Recall that we chose the form $B_\k$ so that $C^{B_\k}_\k$ acts as the identity on $\p$, thus we can write
\begin{align*}
\sum_{i,j,r} X^2_i(v_j\wedge v_r)\wedge v_j\ \wedge v_r&=\sum_{i,j,r} X^2_i(p_{\k}(v_j\wedge v_r))\wedge v_j\ \wedge v_r+\sum_{i,j,r} X^2_i(p_{\mathfrak p}(v_j\wedge v_r))\wedge v_j\ \wedge v_r\\
&=\sum_{j,r,s} 2g_s^{B_\k}p_{\k_s}(v_j\wedge v_r)\wedge v_j\ \wedge v_r+\sum_{j,r} p_{\mathfrak p}(v_j\wedge v_r)\wedge v_j\ \wedge v_r\\
&=\sum_{j,r,s} (2g_s^{B_\k}-1)p_{\k_s}(v_j\wedge v_r)\wedge v_j\ \wedge v_r+\sum_{j,r} v_j\wedge v_r\wedge v_j\wedge v_r\\
&=\sum_{j,r,s} (2g_s^{B_\k}-1)p_{\k_s}(v_j\wedge v_r)\wedge v_j\ \wedge v_r.
\end{align*}

We now compute $p_{\k_s}(v_j\wedge v_r)$ explicitly. We extend $\langle\cdot,\cdot\rangle$ to $\bigwedge^2V$ by determinants. 
Note that
\begin{align*}\langle \tau^{-1}(X_i),\tau^{-1}(X_j)\rangle&=\frac{1}{16}\sum_{r,s}det\begin{pmatrix}\langle X_i(v_r),X_j(v_s)\rangle& \langle X_i(v_r),v_s\rangle\\\langle v_r,X_j(v_s)\rangle& \langle v_r,v_s\rangle\end{pmatrix}\\
&=\frac{1}{16}\sum_{r}(\langle X_i(v_r),X_j(v_r)\rangle-\langle v_r,X_j(X_i(v_r))\rangle)\\&=-\frac{1}{8}\sum_{r}\langle X_j(X_i(v_r)),v_r\rangle=-\frac{1}{8}tr(X_jX_i).
\end{align*}
 
 Recall that $(X,Y)=\half tr_V(XY)$, hence, if $X_i,X_j \in \k_s$, then $tr_V(X_jX_i)=2(X_j,X_i)=2n_sB_\k(X_j,X_i)$, so that 
 $$\langle \tau^{-1}(X_i),\tau^{-1}(X_j)\rangle=-\frac{1}{4}n_s\delta_{ij};
 $$
  therefore
\begin{align*}
 p_{\k_s}(v_j \wedge v_r)=&-\frac{4}{n_s}\sum_{X_t\in\k_s}\langle v_j\wedge v_r,\tau^{-1}(X_t)\rangle \tau^{-1}(X_t)=-\frac{1}{n_s} \sum_{t,k}\langle v_j\wedge v_r,X_t(v_k)\wedge v_k\rangle \tau^{-1}(X_t)\\
 &= -\frac{1}{n_s}\sum_{t,k}det\begin{pmatrix}\langle v_j,X_t(v_k)\rangle& \langle v_j,v_k\rangle\\\langle v_r,X_t(v_k)\rangle& \langle v_r,v_k\rangle\end{pmatrix} \tau^{-1}(X_t)=- \frac{2}{n_s}\sum_{t}\langle v_j,X_t(v_r)\rangle \tau^{-1}(X_t).
\end{align*}
 
 Substituting we find that
$$
\sum_i \tau^{-1}(X_i)\wedge \tau^{-1}(X_i)=\frac{1}{16}\sum_{s}\frac{2g_s^{B_\k}-1}{n_s}\sum_{j,r=1}^{\dim V}(\sum_{X_t\in \k_s} \langle v_j,X_t(v_r)\rangle \tau^{-1}(X_t)\wedge v_j\wedge v_r)
$$
By \eqref{ennej}, we have $\frac{2g_s^{B_\k}-1}{n_s}=-2k$, hence
\begin{align*}
\sum_i \tau^{-1}(X_i)\wedge \tau^{-1}(X_i)&=-\frac{k}{32}\sum_{s}\sum_{j,r=1}^{\dim V}(\sum_{X_t\in \k_s} \langle v_j,X_t(v_r)\rangle \tau^{-1}(X_t)\wedge v_j\wedge v_r)\\
&=-\frac{k}{32}\sum_{j,r=1}^{\dim V}(\sum_{t=1}^{\dim \k} \langle v_j,X_t(v_r)\rangle \tau^{-1}(X_t)\wedge v_j\wedge v_r)\\
&=-\frac{k}{32}\sum_{j,r,t,k} \langle v_j,X_t(v_r)\rangle X_t(v_k)\wedge v_k\wedge v_j\wedge v_r\\
&=-\frac{k}{32}\sum_{r,t,k}  X_t(v_k)\wedge v_k\wedge X_t(v_r)\wedge v_r\\&=-\frac{k}{2}\sum_t \tau^{-1}(X_t)\wedge \tau^{-1}(X_t).
\end{align*}

Since  $k\in\ganz_+$, we have $\sum_t \tau^{-1}(X_t)\wedge \tau^{-1}(X_t)=0$ as desired.
\end{proof}

\begin{lemma}\label{invariants}
If $\widetilde V(k,\k)$ embeds conformally in $V_k(so(V))$ with $k\in \ganz_+$,  then
$V^\k=\{0\}$.
\end{lemma}
\begin{proof}
It is well known that $V$ is the irreducible module for $so(V)$ with highest weight the fundamental weight $\omega_1$ (ordering the  simple roots of $so(V)$ in the standard way). Next we notice that if $k\in\ganz_+$, then $L(k\Lambda_0+\omega_1)$ is an integrable $\widehat{so(V)}$--module of level $k$. Since the category of  $V_k(so(V))$--modules coincides with the integrable, restricted $\widehat{so(V)}$--module of level $k$ (cf. \cite{FZ}), we conclude that  $L(k\Lambda_0+\omega_1)$ is a $V_k (so(V))$--module.  
The top component of  $L(k\Lambda_0+\omega_1)$ is isomorphic to $V$ as a $so(V)$-module. Since the vectors in $V^\k$ are $\widetilde V(k,\k)$-singular vectors in  $L(k\Lambda_0+\omega_1)$, $(\omega_\k)_{0}$ acts trivially on them. Since $(\omega_{so(V)})_0$ acts as 
$$ \frac{ \langle \omega_1 , \omega _1 + 2 \rho \rangle }{2 (k + h^{\vee} )} I_V = \frac{\dim V-1}{2(k+\dim V-2)}I_V$$ on $V$  we have that either $\dim V=1$ or $V^\k=\{0\}$. The former possibility is excluded since it implies $so(V)=\{0\}$, thus $\k=\{0\}$. We are assuming that $K$ is connected and non-trivial, hence this is not possible.
\end{proof}



We are now ready to prove the converse implication of the Symmetric Space Theorem:
\begin{prop}
Assume that $\widetilde V(k,\k)$ embeds conformally in $V_k(so(V))$ with $k\in \ganz_+$. 
Then $k=1$, there is a Lie algebra structure on $\mathfrak r=\k\oplus V$ making $\mathfrak r$ semisimple, $(\mathfrak r,\mathfrak k)$ is a symmetric pair, and a nondegenerate invariant form on $\mathfrak r$ is given by the direct sum of $B_\k$ and $\langle\cdot,\cdot\rangle$.
\end{prop}
\begin{proof}In \cite[Proposition 1.37]{Kost}
 it is shown that  $\nu_*(C_\k)$ might have nonzero components only in degrees $0,4$ w.r.t. the standard grading of $\bigwedge V \cong Cl(V)$.
Recall that, if $y\in V$ and $w\in Cl(V)$, then $y\cdot w=y\wedge w+i(y)w$, hence 
 $\nu_*(C_\k)=\sum_i \tau^{-1}(X_i)\wedge \tau^{-1}(X_i)+a$ with  $a\in\wedge^0V=\C$. 
By Lemma \ref{critKost}, we have that $\sum_t \tau^{-1}(X_t)\wedge \tau^{-1}(X_t)=0$, hence $\nu_*(C_\k)\in\C$. By \cite[Theorem 1.50]{Kost}, $\mathfrak r=\k\oplus V$ has the structure of a Lie algebra and $B_\k\oplus\langle\cdot,\cdot\rangle$ defines a nondegenerate invariant form on $\mathfrak r$. Moreover, by \cite[Theorem 1.59]{Kost}, the pair $(\mathfrak r,\mathfrak k)$ is a symmetric pair. By Lemma \ref{invariants} we have $V^\k=\{0\}$. We can therefore apply \cite[Theorem 1.60]{Kost} and deduce that $\r$ is semisimple.
Finally, we have to show that $k=1$. If the embedding is conformal then $\omega_{so(V)}-\omega_\k$ is in the maximal ideal of  $V^k(so(V))$, hence there must be a singular vector  in $V^k(so(V))$ of conformal weight $2$. This implies that there is  $\l$, with $\l$ either zero or a sum of at most two roots of $so(V)$, such that
$$
\frac{(2\rho_{so(V)}+\l,\l)}{2(k+\dim V-2)}=2.
$$
Here $\rho_{so(V)}$ is a Weyl vector for $so(V)$. The previous equation can be rewritten as
$$
\frac{(2\rho_{so(V)}+\l,\l)-2(k+\dim V-2)}{k+\dim V-2}=2.
$$
 Since  $(\rho_{so(V)},\a)\le \dim V-3$  for any root $\a$ and $\Vert \l\Vert^2\le 8$, we see that the above equality implies that $
\frac{2(\dim V-2)+4-2k}{k+\dim V-2}\ge 2$, so $k\le 1$.

\end{proof}

\section{Classification of equal rank conformal embeddings}\label{ERE}
It is natural to investigate conformal embeddings beyond the integrable case. 
  Recall \cite[Theorem 3.1]{AKMPP1} that if $\k$ is maximal equal rank in $\g$, then 
\begin{equation}\label{c}
\widetilde V(k,\k) \text{Ê is conformal in } V_k(\g)\iff c(\k)=c(\g) 
\end{equation} Here we provide a complete classification of conformal embeddings in the case when $rk(\k)=rk(\g)$,
making more precise both the statement and the proof of \cite[Proposition 3.3]{AKMPP1}.

  That result gives a numerical criterion to reduce the detection of  conformal embeddings from any subalgebra
  $\k$ reductive in $\g$ to a maximal one:
let 
  \begin{equation}\label{chain}
  \k=\k_0\subset\k_1\subset \k_2\subset\dots \subset\k_t=\g\end{equation} be a sequence of equal rank subalgebras with $\k_i$ maximal in $\k_{i+1}$. 
  Let $\k_i=\oplus_{j=0}^{n_i}\k_{i,j}$ be the decomposition of $\k_i$ into simple ideals 
  $\k_{i,j},\,j\geq 1$ and a center $\k_{i,0}$. Since $\k_{i-1}$ is maximal and equal rank in $\k_i$, 
  there is an index $j_0\geq 1$ such that 
 $\k_{i-1}=\oplus_{j\ne j_0}\k_{i,j}\oplus\tilde \k_{i-1}$ with $\tilde\k_{i-1}$ maximal in $\k_{i,j_0}$ (note that $\tilde\k_{i-1}$ is not simple, in general).  \par

  \begin{theorem}With notation as above, 
$\widetilde V(k,\k)$ is conformally embedded in $V_k(\g)$
 if and only if for any $i=1,\ldots,t$ we have
\begin{equation}\label{non maximal}
c(\tilde \k_{i-1})=c(\k_{i,j_0}).\end{equation}
\end{theorem}
 \begin{proof} 
 Assume first that $\widetilde V(k,\k)$ is conformally embedded in $V_k(\g)$.  We prove \eqref{non maximal} by induction on $t$. The base $t=1$ corresponds to $\k$ maximal in $\g$. Since $\omega_{\k}=\omega_{\g}$ it is clear that $c(\k)=c(\g)$ which is condition \eqref{non maximal} in this case.
 
 Assume now $t>1$. Since the form $(\cdot,\cdot)$ is nondegenerate when restricted to $\k_{t-1}$, the orthocomplement $\p$  of $\k$ in $\g$ can be written as $\p=\p\cap\k_{t-1}\oplus V$ with $V$ the orthocomplement of $\k_{t-1}$ in $\g$. Let $\overline V(k,\k)$ be the vertex subalgebra of $V_k(\k_{t-1})$ generated by $a_{(-1)}\vac$, $a\in\k$.  Since $\omega_\k=\omega_\g$, we have that $\omega_{\k}x_{(-1)}\vac=x_{(-1)}\vac$ for all $x\in \p\cap\k_{t-1}$, hence, by Theorem \ref{AP}, $\overline V(k,\k)$ is conformally embedded in $V_k(\k_{t-1})$, thus, by the induction hypothesis, \eqref{non maximal} holds for $i=1,\ldots,t-1$. In particular we have $c(\k)=c(\k_{t-1})$. Since $c(\k)=c(\g)$, we get \eqref{non maximal} also for $i=t$.
 \par


We show on case by case basis that condition \eqref{non maximal} is sufficient.\par\vskip5pt
 {\bf Type $A$.} Recall the following facts.
 \begin{enumerate}
 \item A maximal equal rank reductive subalgebra of $A_n$ is of type $A_{h-1}\times A_{n-h}\times Z$ with $Z$ a one-dimensional center.
 \item Possible  non-integrable conformal embeddings occurr at level $-1$ except when $n=1$ or $h=n-1$ and 
at level $-\frac{n+1}{2}$ except when $n=1$ or $h=n-h-1$ \cite[Theorem 5.1 (1)]{AKMPP1}.
 \item $\widetilde V_{-1}(A_{h-1}\times A_{n-h}\times Z)$ is simple if $n>5, h>2, n-h>1$ \cite[Theorem 5.1 (1)]{AKMPP1}.
 \item  $\widetilde V_{-1}(A_{2}\times A_{2}\times Z)$ is simple \cite[Theorem 5.2 (1)]{AKMPP1}.
 \item $\widetilde V_{-1}(A_{h-1}\times A_{n-h}\times Z)$ is simple if $h=2$ or  $n-h=1$  (see proof of 
 Theorem 5.3 (2) in \cite{AKMPP1}).
 \end{enumerate}
 Let $\k=\k_0\subset\k_1\subset \k_2\subset\dots \subset\k_t=\g$ be a chain as in \eqref{chain}. It is clear that \eqref{non maximal} implies $c(\k_i)=c(\k_{i-1})$. In particular
 $c(\k_{t-1})=c(\g)$, hence by (2) either $k=-\frac{n+1}{2}$ or $k=-1$. In the former case we claim that $t=1$. If by contradiction $t>1$, then  $c(\tilde\k_{t-2})=c(\k_{t-1,j_0})$. By (1), we know that $\k_{t-1}$ is of type $A_{r-1}$, hence $k=-1$ or $k=-r/2$. In turn $\frac{n+1}{2}=-1$ or $\frac{n+1}{2}=-r/2$. Both cases are impossible, since 
 $n>1$. Therefore $t=1$ and by    \eqref{c} $\widetilde V_{k}(\k_{t-1})$ embeds conformally in $V_k(\g)$.\par
Assume now that $k=-1$. Again by \eqref{c} we have that $\widetilde V_{k}(\k_{t-1})$ embeds conformally in $V_k(\g)$. But (3)-(5) imply that  $\widetilde V_{k}(\k_{t-1})= V_{k}(\k_{t-1})$ and we can proceed inductively.\par
\vskip5pt
{\bf Types $B, D$.} Recall that the maximal equal rank reductive  subalgebras of $so(n)$ are  $so(s)\times so(n-s)$ and $sl(n)\times Z\hookrightarrow so(2n)$. The only non-integrable conformal level is $2-n/2$ in the first case and $-2$ in the second. 
\par
Assume $n\ge 7, n\ne 8$. Consider a chain as in \eqref{chain}. The condition  $c(\k_{t-1})=c(\g)$ implies that $k=2-n/2$ or 
$\k_{t-1}$ is of type $A_{n/2-1}\times Z,$ $n$ even and $k=-2$. We claim that in both cases $t=1$.   If by contradiction $t>1$, then  $c(\tilde\k_{t-2})=c(\k_{t-1,j_0})$. Assume  $\k_{t-1}=so(s)\times so(n-s)$, we may assume that 
$\tilde\k_{t-1,j_0}$ is a simple component of $so(s)$; in particular $s\ge 3$. If  $s=3$  then we have equality of central charges only at level $1$, which implies $n=0$. If $s=4$ then 
$so(4)=sl(2)\times sl(2)$ and again $1=2-n/2$. If $s=5$ we have $-1/2=2-n/2$ hence $n=5$. If $s=6$, the level is either $-1$ or $-2$; correspondingly $2-n/2=-1$, i.e. $n=6$ of $2-n/2=-2$, i.e. $n=8$. If $s\ge 7$, then the level is 
$2-s/2=2-n/2$, which gives $s=n$, or $\tilde\k_{t-2}=sl(s/2)\times Z$. In this last case $-2=2-n/2$, or $n=8$.
So $t=1$. \par
In the other case $\tilde\k_{t-2}\subset A_{n-1}$; since the central charges are equal, we should have either
$-2=-1$ or $-n/2=-2$, which is not possible. So again $t=1$ and we can finish the proof using \eqref{c}.\par


Finally, in case $n=8$, we observe that the only possible chain is
\begin{equation}\label{dett}
A_2\times Z_1\times Z\subset A_3\times Z\subset D_4
\end{equation}
with $k=-2$ and $Z_1$, $Z$ one-dimensional subalgebras. We need only to check that the hypothesis of Theorem \ref{AP} hold in this case. We use the following setup: 
from the explicit realization
$$so(8,\C)=\left\{\begin{pmatrix} A & B \\ C & -A^t\end{pmatrix}\mid A,B,C\in gl(4,\C),\, B=-B^t,\,C=-C^t\right\}.$$
we have 
\begin{align*}
sl(4)\times Z&=\left\{\begin{pmatrix} A & 0 \\ 0 & -A^t\end{pmatrix}\mid A\in gl(4,\C)\right\}=\left\{\begin{pmatrix} A & 0 \\ 0 & -A^t\end{pmatrix}\mid A\in sl(4,\C)\right\}\oplus \C \begin{pmatrix} I & 0 \\ 0 & -I\end{pmatrix}
\\
sl(3)\times Z_1 &=\left\{\begin{pmatrix} A &0& 0&0 \\ 0& -tr(A) &0&0\\0&0& -A^t&0\\0&0&0&tr(A)\end{pmatrix}\mid A\in gl(3,\C)\right\}
\\&=\left\{\begin{pmatrix} A &0& 0&0 \\ 0& -tr(A) &0&0\\0&0& -A^t&0\\0&0&0&tr(A)\end{pmatrix}\mid A\in sl(3,\C)\right\}\oplus \C \begin{pmatrix} I &0& 0&0 \\ 0& -3 &0&0\\0&0& -I&0\\0&0&0&3\end{pmatrix}
\end{align*}

Let $(\cdot,\cdot)$ be the normalized invariant form of $so(8)$. With notation as in \eqref{numcheck}, in this case, we have $\k_1=sl(3)$, $\k_0=Z\times Z_1$, $(\cdot,\cdot)_1=(\cdot,\cdot)_{|sl(3)\times sl(3)}$ and, by definition, $(\cdot,\cdot)_0=(\cdot,\cdot)_{|\k_0\times \k_0}$. 

 Let $\varepsilon_i\in\h^*$ be defined by setting $\varepsilon_i(E_{j,j}-E_{4+j,4+j})=\d_{i,j}$, $1\le i,j\le 4$. Set $\eta\in Z^*$ be defined by setting $\eta(\begin{pmatrix} I & 0 \\ 0 & -I\end{pmatrix})=1$ and $\eta_1\in Z_1^*$ be defined by $\eta_1(\begin{pmatrix} I &0& 0&0 \\ 0& -3 &0&0\\0&0& -I&0\\0&0&0&3\end{pmatrix})=1$. 
It is clear that, as  $sl(4)\times Z$-module, $so(8,\C)=(sl(4,\C)\times Z)\oplus ( \bigwedge^2 \C^4)\oplus (\bigwedge^2 \C^4)^*$. As 
$gl(3)\times Z$-module we have
$$
  \bigwedge^2 \C^4=\bigwedge^2 \C^3\otimes V(2\eta)\oplus \C^3\otimes V(2\eta)
  $$
  and
  $$
  ( \bigwedge^2 \C^4)^*=(\bigwedge^2 \C^3)^*\otimes V(-2\eta)\oplus (\C^3)^*\otimes V(-2\eta) 
  $$
As $sl(3)\times Z_1\times Z$-module, we have, setting $\omega_1=\tfrac{2}{3}\varepsilon_1-\tfrac{1}{3}\varepsilon_2-\tfrac{1}{3}\varepsilon_3$ and $\omega_2=\tfrac{1}{3}\varepsilon_1+\tfrac{1}{3}\varepsilon_2-\tfrac{2}{3}\varepsilon_3$
  $$
  \bigwedge^2 \C^4=V_{sl(3)}(\omega_2)\otimes V(2\eta_1)\otimes V(2\eta)\oplus V_{sl(3)}(\omega_1)\otimes V(-2\eta_1)\otimes V(2\eta)
  $$
  $$
  ( \bigwedge^2 \C^4)^*=V_{sl(3)}(\omega_1)\otimes V(-2\eta_1)\otimes V(-2\eta)\oplus V_{sl(3)}(\omega_2)\otimes V(2\eta_1)\otimes V(-2\eta)
   $$
   
We extend $\eta$ and $\eta_1$ to $\h$ by setting
 $$
 \eta(\begin{pmatrix} I &0& 0&0 \\ 0& -3 &0&0\\0&0& -I&0\\0&0&0&3\end{pmatrix})=\eta_1(\begin{pmatrix} I & 0 \\ 0 & -I\end{pmatrix})=0,\quad \eta_{|\h\cap sl(3)}=(\eta_1)_{|\h\cap sl(3)}=0.
 $$ 
 Explicitly
 $$
 \eta=\frac{1}{4}\sum_{i=1}^4\varepsilon_i,\quad \eta_1=\frac{1}{12}\sum_{i=1}^3\varepsilon_i-\frac{1}{4}\varepsilon_4.
 $$
 Since $\begin{pmatrix} I & 0 \\ 0 & -I\end{pmatrix}$ and $\begin{pmatrix} I &0& 0&0 \\ 0& -3 &0&0\\0&0& -I&0\\0&0&0&3\end{pmatrix}$ are orthogonal and $(Z\times Z_1)^\perp=\h\cap sl(3)$, 
 then, for $\lambda=x\eta+y\eta_1,\mu=x'\eta+y'\eta_1\in (Z\times Z_1)^*$, we have
\begin{equation}
(\l,\mu)_0=(x\eta+y\eta_1,x'\eta+y'\eta_1).
\end{equation}
We are now ready to apply  Theorem \ref{AP} and check \eqref{numcheck}. Consider the component $V_{sl(3)}(\omega_2)\otimes V(2\eta_1)\otimes V(2\eta)$ of $\p$. Then, in this case, $\rho_0^1=2\epsilon_1-2\epsilon_3$ and $\rho_0^0=0$ so that
$$
\sum_{j=0}^1\frac{ ( \mu_{i}^j , \mu _{i}^j + 2 \rho_0^j ) _j}{ 2 (k_j + h_j
^{\vee} )} =\frac{ ( \omega_2 , \omega_2+ 2 \rho_0^1 ) }{ 2 (-2+ 3 )} +\frac{ ( 2\eta_1+2\eta , 2\eta_1+2\eta ) }{-4}=4/3 -1/3=1.$$
The other components of $\p$ are handled by similar computations. It follows from Theorem \ref{AP} that $sl(3)\times Z_1\times Z_2$ embeds conformally in $D_4$.

\vskip5pt
{\bf Type $C$.} Recall that the maximal equal rank reductive  subalgebras  are of type $C_h\times C_{n-h},\,
h\geq 1, n-h\geq 1$ or $A_{n-1}\times Z$. In both cases then  level $-1/2$ occurs, whereas in the first case we have also level $-1-n/2,\,h\ne n-h$.
Reasoning as in the previous cases, one deals with level $-1-n/2$. For the other level we can use Theorems 2.3, 5.1, 5.2 (2) from \cite{AKMPP1} to obtain the simplicity of $\tilde V(\k_i)$.
\vskip5pt
{\bf Exceptional types.}  In these cases, starting with a chain of subalgebras \eqref{chain}, we directly check that either
$t=1$, so that we can conclude by   \eqref{c},  or $t=2$ and we are in one of the following cases:
\begin{align*}
&A_1\times A_1\times A_1\times D_4\hookrightarrow A_1\times D_6\hookrightarrow E_7\quad&k=-4,\\
&A_1\times D_5\times Z\hookrightarrow A_1\times D_6\hookrightarrow E_7\quad&k=-4,\\
&A_1\times A_4\times Z\hookrightarrow A_1\times A_5\hookrightarrow E_6\quad&k=-3,\\
&A_1\times A_1\times A_3\times Z\hookrightarrow A_1\times A_5\hookrightarrow E_6\quad&k=-3,\\
&D_4\times Z \times Z\hookrightarrow D_5\times Z\hookrightarrow E_6\quad&k=-3,\\
&A_1\times A_1\times A_3\times Z\hookrightarrow D_5\times Z\hookrightarrow E_6\quad&k=-3,\\
&A_1\times A_1\times B_2\hookrightarrow A_1\times C_3\hookrightarrow F_4\quad&k=-5/2.\\
\end{align*}
We can then check the conformality of each composite embedding using the AP-criterion as done for \eqref{dett}.
 \end{proof}
 
 \begin{defi}
 
 We say that a chain of vertex algebras
 $U_1 \subset U_2 \subset \cdots \subset U_n$
 is conformal if $U_{i} $ is conformally embedded into $U_{i+1}$ for all $i=1, \dots , n-1$.
 \end{defi}
 
 \begin{rem} \label{new44}\ 
\begin{enumerate}
\item The conformality of the embedding given in \eqref{dett} can be derived in a different way using Lemma 10.4 of \cite{AKMPP3}.  More precisely,  result from \cite{AKMPP3} implies that  the vertex subalgebra $\widetilde V (-2, sl(4) \times Z)  $ of $V_{-2} (D_4)$ must be  simple. Since  $\widetilde V (-2, sl(3)  \times Z)  $  is a simple vertex subalgebra of $\widetilde V (-2, A_3 )  $  we actually have conformal chain of  {\sl simple} affine vertex algebras  
\bea  V_{-2} (A_2 ) \otimes M(1) \otimes M(1) \subset  V_{-2} ( A_3) \otimes M(1) \subset V_{-2} (D_4), \label{chainn} \eea
where $M(1)$ is the Heisenberg vertex algebra of rank one.

The determination of branching rules for conformal embeddings in \eqref{chainn} is an important open problem which is also recently  discussed  in the physics literature \cite[Section 5]{Ga}.
\item Similar  arguments can be applied in the cases $A_1\times A_1\times A_3\times Z\hookrightarrow D_5\times Z\hookrightarrow E_6$;  $D_4\times Z \times Z\hookrightarrow D_5\times Z\hookrightarrow E_6$ at $k=-3$. It was proved  in  \cite{AP-JAA}  that   $\widetilde V (-3 , D_5 \times Z\ )  $ is a  simple vertex algebra  and therefore, by equality of central charges,  we have  conformal chains  of vertex algebras
$$\widetilde V( -3, A_1\times A_1\times A_3 ) \otimes M(1) \subset V_{-3} (D_5) \otimes M(1) \subset V_{-3} (E_6),  $$
$$\widetilde V(-3, D_4 \times Z  ) \otimes M(1) \subset V_{-3} (D_5) \otimes M(1) \subset V_{-3} (E_6). $$
\item There are however cases where conformal embeddings in non-simple vertex algebras occur. This is quite possible (see \cite[Remark 2.1]{AKMPP3}).
One instance of this phenomenon occurs in $E_7$. By using similar arguments as in \cite[Section 8]{AKMPP3} one can show that the affine vertex subalgebra $\widetilde V (-4 , A _1\times D_6\ )  $  of  $V_{-4} (E_7) $ is not simple. We however checked  the conformality of the chains
$$\widetilde V(-4,A_1\times D_5\times Z ) \hookrightarrow \widetilde V (-4 , A _1\times D_6\ ) \hookrightarrow V_{-4}(E_7),$$
$$
\widetilde V(-4,A_1\times D_5\times Z) \hookrightarrow \widetilde V (-4 , A _1\times D_6\ ) \hookrightarrow V_{-4}(E_7) $$ using AP criterion.
\end{enumerate}
 \end{rem} 
 \section{Embeddings  at the critical level}\label{CL}
 \label{embd-crit}
 
 The classification of maximal conformal embeddings was studied in detail in \cite{AKMPP3}. On the other hand we detected in \cite{AKMPP3} some border cases where we can not speak of conformal embeddings since  the  embedded affine vertex subalgebras have critical levels. In this section we provide some results on this case. Instead of considering  equality of confomal vectors, we shall here consider  equality of central elements.  We prove that the embedded vertex subalgebra is the quotient of
 $( V^{-h^{\vee} } ({\g}_1 ) \otimes  V^{-h^{\vee}} ({\g}_2)    ) \slash   \langle s  \rangle $ where $s$ is a linear combination  of quadratic Sugawara central elements. 
 In particular  the embedded vertex subalgebra is not the tensor product of two affine vertex algebras. 
 
 It seems that the embeddings at the critical levels  were also investigated in physical literature. In particular, Y. Tachikawa presented in \cite{T} conjectures on existence of certain embeddings of affine vertex algebras at the critical level inside of larger vertex algebras.

 \subsection{ A criterion }
 We shall describe a   criterion for embeddings at the critical level.  
 \begin{proposition} \label{critical-criterion}
Assume that %
 %
there is a central element $s \in \bigotimes_jV^{k_j}  (\k_j) $ that  is not a scalar multiple of ${\vac}$ and, in $V^k(\g)$, 
$$ s_{(n)}  v  = 0 \quad \text{ for all } v \in  {\mathfrak p } \text{ and }n \ge 1. $$
 Then $s = 0$ in $V_k(\g)$. Moreover there is a non-trivial homomorphism of vertex algebras
 $$ \Phi^{(s)} :  \left(\bigotimes_jV^{k_j}  (\k_j) \right) \bigg/  { \langle s\rangle } \rightarrow V_k(\g).$$
 \end{proposition}
 \begin{proof}
 We need to prove that $s$ belongs to the maximal ideal in $V^k (\g)$. It suffices to prove that 
\bea \label{singular} x_{(n)} s = 0,\ n>0, \ x\in\g. \eea
  Since $s$ is in the center of $\bigotimes_jV^{k_j}  (\k_j) $, we get that
 $$ x_{(n)} s = 0 \quad  \text{ for all }  x \in \k, \ n \ge 0. $$
 Assume now that $x \in {\mathfrak p}$ and $n\ge 1$.  We have,
  $$ x_{(n)} s  = - [s_{(-1)}, x_{(n)}] {\bf 1} = - \sum_{j\ge0}\binom{-1}{j}(s_{(j)}x)_{(n-1-j)} {\bf 1} =-(s_{(0)}x)_{(n-1)} {\bf 1}= 0.  $$
 The claim follows.
 \end{proof}

\subsection {Embeddings for Lie algebras of type $A$}\label{AA}

Let $M_ {\ell} $ be the Weyl vertex algebra.
%
%
Recall that  the sympletic affine vertex algebra $V_{-1/2} (C_{\ell} ) $ is a subalgebra of $M_{\ell}$. This realization provides a nice framework for studying embeddings of affine vertex algebras in $V_{-1/2} (C_{\ell}) $. Some interesting cases were studied in \cite{AKMPP3}. In this section we study the embeddings  in $M_{\ell}$ of critical level affine vertex algebras. 
 
 Consider the case $\ell = n^2$. Then the embeddings $sl(n)\times sl(n)\hookrightarrow sl(\ell)$, $sl(n)\times sl(n)\times \C\hookrightarrow gl(\ell)$ induce  vertex algebra  homomorphisms   \bea \Phi :  V^{-n} (sl(n) ) \otimes V^{-n} (sl(n ) ) & \rightarrow &V _{-1}(sl(\ell ) ) \subset M_{\ell}, \nonumber   \\
\overline \Phi :  V^{-n} (sl(n) ) \otimes V^{-n} (sl(n ) ) \otimes M(1)  & \rightarrow &V _{-1}(gl(\ell ) ) \subset M_{\ell}.  \nonumber \eea

 We first want  to describe $\text{Im} (\Phi)$.  We shall see that, quite surprisingly,  $\mbox{Im} (\Phi)$ does not have the form $\widetilde V^{-n} (sl(n) ) \otimes \widetilde V^{-n} (sl(n ) )$, where  $\widetilde V^{-n} (sl(n) ) $ is a certain quotient of $V^{-n} (sl(n) )$.

Let $ \{x_i\} $, $ \{y_i\} $ be dual bases of $sl(n)$ with respect to the trace form, and let
$$ S=\sum_{i=1} ^{n^2-1} : x_i y_i: .
$$Then $S  $ is a central element of  the vertex algebra $V^{-n} (sl(n))$. Define $S_1 = S\otimes 1$, $S_2 = 1\otimes S$, so that
$$s := S_1-S_2\in V^{-n}(sl(n))\otimes V^{-n}(sl(n)).
$$

\begin{lemma}  \label{critical-emb}We have:
\begin{itemize}
\item[(1)] $\Phi(S_1) =\Phi(S_2)$ in $V_{-1} (sl(\ell)) \subset  M_{\ell}$.
\item[(2)]  $\Phi(S_1 )\ne 0$ in $V_{-1} (sl(\ell))$.
\end{itemize}

\end{lemma}
\begin{proof}
First we notice that
$sl(\ell)  = sl(n) \times sl(n) \bigoplus  {\mathfrak p}$ where $\mathfrak{p} \simeq L_{sl(n)} (\theta) \otimes L_{sl(n)} (\theta) $ as a $sl(n)\times sl(n)$-module

Since  the quadratic Casimir operator acts  on $L_{sl(n)} (\theta)$ as
$$(\theta,\theta+2\rho)_{sl(n )}=2 n \ne 0, $$
we get, for $x \in \mathfrak p$,
$$ (S_i )_{(1)} x_{(-1)}{\bf 1} =  2n  x_{(-1)}{\bf 1}  \qquad (i=1,2). $$
This shows that $\Phi(S_i) \ne 0$ in $V_{-1} (sl(\ell))$. Then  $s = S_1 -S_2$ is a central element of  $V^{-n} (sl(n) ) \otimes  V^{-n} (sl(n ) )$ which satisfies  the conditions of Proposition  \ref{critical-criterion}. So $s = 0$ in $V_{-1} (sl(\ell)) $. The claim follows.
\end{proof}

\begin{remark}
 Note that the proof of  Lemma \ref{critical-emb}  is similar to the proof of  the AP  criterion. But instead of proving  the equality of conformal vectors, we prove  the equality of central elements.
\end{remark}
 \vskip5pt

Let $$\widetilde V(sl(n) \times sl(n) ) =\left( V^{-n} (sl(n) ) \otimes V^{-n} (sl(n ) )  \right)\big/ {    \langle s  \rangle }, $$
where $ \langle s \rangle$ is ideal in  $V^{-n} (sl(n) ) \otimes V^{-n} (sl(n ) )  $ generated by $s$.
Using Lemma \ref{critical-emb} we get:
\begin{prop} \label{str-A} We have:
\begin{itemize} 
\item[(1)] $ \text{\rm Im} (\Phi)$ is  a quotient of $\widetilde V(sl(n) \times sl(n) )$.
\item[(2)]   $ \text{\rm Im} (\Phi)$ is not simple.
\end{itemize}
\end{prop}

\subsection{A conjecture}
We will now present a conjecture on $ \mbox{Im} (\Phi)$.

Recall that the center of $V^{-n} (sl(n))$ (the  Feigin-Frenkel center)  is generated  by the central element 
$ S^{(2)} = S$ and by the higher rank Sugawara elements
$S^{(3)} , \dots, S^{(n)}$.
Define the following quotient of $ V^{-n} (sl(n) ) \otimes V^{-n} (sl(n ) ) $:
$$ \mathcal V(sl(n) \times sl(n) ) = \left(  V^{-n} (sl(n) ) \otimes V^{-n} (sl(n ) )\right) \big/    \langle s_2, s_3, \cdots, s_n  \rangle , $$
where
$s_k = S^{(k)} \otimes {\bf 1} - {\bf 1} \otimes S^{(k)}. $

\begin{conj}
$ \mbox{\rm Im} (\Phi) \cong \mathcal V(sl(n) \times sl(n) ) $.
\end{conj}

\subsection{ Maximal embeddings in other simple Lie algebras.}  We now investigate whether similar phenomena as the one discussed in Section \ref{AA} can occur for other types. 

We list the maximal embeddings $\k  \hookrightarrow   \g$ with $\k$ semisimple but not simple such that there is $k\in\C$ with $k_j+h^\vee_j=0$ for all $j$:
\begin{enumerate}
\item $sp(2n) \times sp(2n) \hookrightarrow   so (4 n ^2)$ at level $k = -1-1/n$; 
\item  $so(n) \times so(n) \hookrightarrow   so (n ^2)$ at level $k = -1+2/n$;
\item $sp(2 n) \times so(2 n +2 ) \hookrightarrow   sp ( 4n (n+1) )$ at level $k = -1/2$; 
\item $so(n) \times so(n ) \hookrightarrow   so ( 2n) $ at level $k = 2- n$;
\item $sp(2n) \times sp(2n ) \hookrightarrow   sp (4n) $ at level $k = -1- n$;
 \item   $sl(5) \times sl(5 ) \hookrightarrow   E_8$ at level $k = -5$;
\item $sl(3)\times sl(3) \times sl(3 ) \hookrightarrow   E_6$ at level $k = -3$;
\end{enumerate}

Some of them  were detected in \cite[Theorem 4.1]{AKMPP3}. In this situation we have a
 homomorphism of vertex algebras
$$ \Phi : \bigotimes_jV^{-h^{\vee}_j } ({\k}_j )     \rightarrow V_{k}  (\g). $$

Write $\k=\oplus_{j=1}^t\k_i$ with $\k_i$ simple ideals. Let $\{x^j_i\},\{y^j_i \}$ be dual bases of $\k_j$ with respectt to $(\cdot,\cdot)_j$ and let
Let $S_j=\sum_i:x_i^jy_i^j:$ be the quadratic Sugawara central element in $V^{-h^{\vee}_j} ({\k}_j)$. Set $s_j=\vac\otimes \cdots \otimes S_j\otimes \cdots \otimes \vac$. Write 
$\p\oplus_i V_i$ for the complete decomposition of $\p$ as a $\k$-module. If $V_i=\otimes_j L_{\k_j}(\mu_{ij})$, set $\l_{ij}=(\mu_{ij},\mu_{ij}+\rho_0^j)_j$.

\begin{prop}\label{IM}If $z$ is a non zero vector  $(z_1,\cdots,z_t)\in\C^t$ such that $\sum_j\l_{ij}z_j=0$ for all  $i$,  set $s_z = \sum_j z_js_j$
and
 $$\widetilde V(\k) =\left(\otimes_j V^{-h^{\vee}_j } ({\k}_j )\right)\bigg/{    \langle s_z  \rangle }. $$
Then
$ \mbox{\rm Im} (\Phi)$ is  a non-simple  quotient of $\widetilde V(\k )$.
\end{prop}
\begin{proof}
Note that for all $j$ there is $i$ such that $\l_{ij}\ne 0$, for, otherwise, $\k_j$ would be an ideal in $\g$. It follows that $\Phi(s_j)\ne0$, so $ \mbox{\rm Im} (\Phi)$ is not simple. 
Since $s$ satisfies the hypothesis of Lemma \ref{critical-emb}, the proof follows.
\end{proof}

We now apply Proposition \ref{IM} to the cases listed above.
  
In case (1), as shown in the proof of Theorem 4.1 of \cite{AKMPP3} we have, as $sp(2n)\times sp(2n)$-module,
$$so(4n^2)=S^2(\C^{2n})\otimes \bigwedge^2\C^{2n}\oplus \bigwedge^2\C^{2n}\otimes S^2(\C^{2n})
$$
so 
$$
\p=\begin{cases}(L_{sp(2n)}(\theta)\otimes L_{sp(2n)}(\omega_2))\oplus(L_{sp(2n)}(\omega_2)\otimes L_{sp(2n)}(\theta))&\text{if $n>1$}\\
\{0\}&\text{if $n=0$}
\end{cases}.
$$
Let $\Lambda$ be the matrix $(\l_{ij})$. Then, if $n>1$, 
$$
\Lambda=\begin{pmatrix}2(n+1)&2n\\2n&2(n+1)\end{pmatrix}
$$ so the hypothesis of Proposition \ref{IM} are never satisfied.

Similarly, in case (2),  we have, as $so(n)\times so(n)$-module,
$$so(n^2)=\bigwedge^2\C^{n}\otimes S^2(\C^{n})\oplus S^2(\C^{n})\otimes \bigwedge^2\C^{n}
$$
so 
$$
\p=\begin{cases}L_{so(n)}(\theta)\otimes L_{so(n)}(2\omega_1)\oplus L_{so(n)}(2\omega_1)\otimes L_{so(n)}(\theta)&\text{if $n>4$}\\
L_{A_1}(\theta)^{\otimes3}\otimes \C\oplus L_{A_1}(\theta)^{\otimes2}\otimes \C\otimes L_{A_1}(\theta)\\
\hspace{3cm}\oplus L_{A_1}(\theta)\otimes \C\otimes L_{A_1}(\theta)^{\otimes2}\oplus \C\otimes L_{A_1}(\theta)^{\otimes3}&\text{if $n=4$}\\
L_{A_1}(\theta)\otimes L_{A_1}(2\theta)\oplus L_{A_1}(2\theta)\otimes  L_{A_1}(\theta)&\text{if $n=3$}
\end{cases}.
$$
Then, if $n>4$, 
$$
\Lambda=\begin{pmatrix}2(n-1)&2n\\2n&2(n-1)\end{pmatrix}$$
while, for $n=4$,
$$
\Lambda=\begin{pmatrix}4&4&4&0\\4&4&0&4\\4&0&4&4\\0&4&4&4\end{pmatrix}$$
and
$$
\Lambda=\begin{pmatrix}4&12\\12&4\end{pmatrix}$$
when $n=3$. In all cases the hypothesis of Proposition \ref{IM} are not satisfied.

In case (3) we have
$$sp(4n(n+1))=S^2(\C^{2n})\otimes S^2(\C^{2n+2})\oplus \bigwedge^2\C^{2n}\otimes \bigwedge^2\C^{2n+2}
$$
hence
$$
\p=\begin{cases}L_{sp(2n)}(\theta)\otimes L_{so(2n+2)}(2\omega_1)\oplus L_{sp(2n)}(\omega_2)\otimes L_{so(2n+2)}(\theta)&\text{if $n>2$}\\
L_{sp(4)}(\theta)\otimes L_{A_3}(2\omega_2)\oplus L_{sp(4)}(\omega_2)\otimes L_{A_3}(\theta)&\text{if $n=2$}\\
L_{A_1}(\theta)^{\otimes3}&\text{if $n=1$}
\end{cases}.
$$
 
 In the case  $n=1$,  we have that, if $(z_1,z_2, z_3) )\in\C^3\setminus\{(0,0,0)\}$ with $z_1+ z_2 + z_3 = 0$, then Proposition \ref{IM} holds.
 
 If   $n \ge 2$ then
 $$\L=\begin{pmatrix}2(n+1)&4(n+1)\\2n&4n\end{pmatrix} , $$
thus the hypothesis of Proposition \ref{IM} are satisfied with 
 $s=  s_1 - \frac{1}{2} s_2 $,  which acts trivially on both components of    $\p$.

In cases (4) and (5)  $\p$ is irreducible as $\k$-module, so, since $\k$ is not simple,  the hypothesis of Proposition \ref{IM} are clearly satisfied. 

In case (6) the decomposition of $\p$ as $\k$-module is
$\p=(L_{sl(5)}(\omega_1)\otimes L_{sl(5)}(\omega_3))\oplus (L_{sl(5)}(\omega_2)\otimes L_{sl(5)}(\omega_4))\oplus(L_{sl(5)}(\omega_3)\otimes L_{sl(5)}(\omega_1))\oplus(L_{sl(5)}(\omega_4)\otimes L_{sl(5)}(\omega_2))$ hence
$$
\Lambda=\begin{pmatrix}4&36/5\\36/5&4\\36/5&4\\4&36/5\end{pmatrix}
$$
so the hypothesis of Proposition \ref{IM} are not satisfied.

Finally in case (7) above, $\p=L_{sl(3)}(\omega_1)^{\otimes3}\bigoplus L_{sl(3)}(\omega_2)^{\otimes3} $ hence Proposition \ref{IM} holds: we can take
$s_z=z_1 s_1+z_2 s_2 + z_3s_3$ with $(z_1,z_2, z_3) )\in\C^3\setminus\{(0,0,0)\}$ such that $z_1+ z_2 + z_3 = 0$ and we obtain a  family of non-trivial homomorphisms
 $$ \Phi_ z:  V^{-2} (A_1 ) ^{\otimes 3} / \langle s_z \rangle \rightarrow V_{-3} (E_6). $$
 
 \begin{rem} \label{multi}
  For embedding
$so(n) \times so(n ) \hookrightarrow   so ( 2n) $ at level $k = 2- n$, the case $n=4$ is of particular interest.  We    have embedding $ sl(2) \times sl(2) \times sl(2) \times sl(2)  \hookrightarrow  so(8)$ at $k=-2$. Then  $\p=L_{sl(2)}(\omega_1)^{\otimes4} $ and we can  take 
$$s_z=z_1(S_1\otimes \vac\otimes \vac \otimes \vac )+z_2 (\vac\otimes S_2\otimes \vac \otimes \vac )+ z_3 (\vac\otimes\vac\otimes S_3\otimes \vac ) + z_4 (\vac\otimes\vac\otimes \vac \otimes  S_4) 
$$ 
 for every element $z= (z_1, \dots, z_4) \in {\C} ^{4}  \setminus\{(0,0,0,0)\}$ so that $z_1 + z_2 + z_3 + z_4 = 0$. We have a  family of non-trivial homomorphisms
 $$ \Phi_ z:  V^{-2} (A_1 ) ^{\otimes 4} / \langle s_z \rangle \rightarrow V_{-2} (D_4). $$
 
  Similarly, for the embedding $ sl(2) \times sl(2) \times sl(2)   \hookrightarrow  sp(8)$ at $k=-1/2$,  we get a family of homomorphism   $$V^{-2} (A_1 ) ^{\otimes 3} / \langle   s_z \rangle \rightarrow V_{-1/2} (C_4), $$
  where $z= (z_1, z_2, z_3) \in   {\C} ^{3}  \setminus\{(0,0,0)\}$,  and  $z_1+ z_2 + z_3 = 0$.
\end{rem}

\begin{rem}
Consider the embedding
$$ A_3 \times A_3 \times A_1 \hookrightarrow   D_6 \times A_1 \hookrightarrow   E_7 \quad \mbox{at} \ k= -4. $$
Then embedding  $A_3 \times A_3 \times A_1  \hookrightarrow   E_7 \quad \mbox{at} \ k= -4$ also  gives a subalgebra of $V_{-4} (E_7)$  at the critical level for $A_3 \times A_3$.

\end{rem}

\begin{rem} 
In \cite{T},  \cite{MT} some very interesting conjectures were presented. In particular, it was stated that vertex algebras for $V_{-2} (D_4), V_{-3} (E_6), V_{-4} (E_7), V_{-6} (E_8)$  are related to certain affine vertex algebras at the critical level. In the Remarks above  we have detected subalgebras at the critical level   inside  $ V_{-2} (D_4), V_{-3} (E_6), V_{-4} (E_7)$ which we believe are the vertex algebras appearing in    \cite[Conjecture 5]{T}.

We should say that a proof of the  conjecture from \cite{MT}  was announced in \cite{Ara-2018}. We hope that this new result can be used in describing decomposition of some embeddings at the critical level.

\end{rem}

\vskip3pt
 \footnotesize{
  \noindent{\bf D.A.}:  Department of Mathematics, Faculty of Science, University of Zagreb, Bijeni\v{c}ka 30, 10 000 Zagreb, Croatia;\newline
{\tt adamovic@math.hr}
  
\noindent{\bf V.K.}: Department of Mathematics, MIT, 77
Mass. Ave, Cambridge, MA 02139;
{\tt kac@math.mit.edu}

\noindent{\bf P.MF.}: Politecnico di Milano, Polo regionale di Como,
Via Valleggio 11, 22100 Como,\newline
Italy; {\tt pierluigi.moseneder@polimi.it}

\noindent{\bf P.P.}: Dipartimento di Matematica, Sapienza Universit\`a di Roma, P.le A. Moro 2,
00185, Roma, Italy;\newline {\tt papi@mat.uniroma1.it}

\noindent{\bf O.P.}:  Department of Mathematics, Faculty of Science, University of Zagreb, Bijeni\v{c}ka 30, 10 000 Zagreb, Croatia;\newline
{\tt perse@math.hr}
}

\end{document}